\documentclass[12pt,twoside]{amsart}
\usepackage{amssymb,amsmath,amscd,enumerate,verbatim,xcolor,mathtools,fullpage}

\usepackage[top=3.5cm, bottom=2.5cm, left=3.0cm, right=3cm]{geometry}
\usepackage{color,colortbl, fancyhdr, wrapfig}
\usepackage{multirow}
\usepackage{diagbox}
\usepackage[unicode]{hyperref}
\usepackage{indentfirst}
\usepackage{tikz}
\usepackage{array}
\usetikzlibrary{calc}
\usepackage{mathrsfs}
\usepackage{graphicx}
\usepackage{longtable}
\usepackage{cleveref}
\usepackage[all,cmtip]{xy}

\newtheorem{thm}{Theorem}[section]
\newtheorem{theorem}[thm]{Theorem}
\newtheorem{lemma}[thm]{Lemma}
\newtheorem{proposition}[thm]{Proposition}
\newtheorem{corollary}[thm]{Corollary}
\theoremstyle{definition}
\newtheorem{definition}[thm]{Definition}
\newtheorem{remark}[thm]{Remark}

\newtheorem{question}[thm]{Question}

\newtheoremstyle{mystyle}
{7pt}
{7pt}
{}
{}
{\scshape}
{:}
{.5em}
{}%
\theoremstyle{mystyle}
\newtheorem{claim}{Claim}

\newcommand{\brb}[1]{\left\{#1\right\}}
\newcommand{\abs}[1]{\left| #1\right|}
\newcommand{\tail}[1]{\mathrm{Tail}(#1)}
\newcommand{\ceil}[1]{ \left \lceil #1 \right \rceil}
\newcommand{\flo}[1]{\left \lfloor #1 \right\rfloor}

\hypersetup{
	colorlinks=true,       
	linkcolor=red,          
	citecolor=blue,        
	filecolor=magenta,      
	urlcolor=cyan           
}

\begin{document}

\title[Independence polynomials and WLP for tadpole graphs]
{Independence polynomials and the weak Lefschetz property for tadpole graphs}

\author[T.Q. Hoa]{Tran Quang Hoa}
\address{University of Education, Hue University,  34 Le Loi St., Hue City, Vietnam.}
\email{tranquanghoa@hueuni.edu.vn}

\author[N.D. Phuoc]{Nguyen Duy Phuoc}
\address{University of Education, Hue University, 34 Le Loi St., Hue City, Vietnam}
\email{ndphuoc@dhsphue.edu.vn}

\author[T.N.T. Son]{Tran Nguyen Thanh Son}
\address{University of Education, Hue University, 34 Le Loi St., Hue City, Vietnam}
\email{tntson@dhsphue.edu.vn}

\date{\today}

\begin{abstract}
	Let $T_{m,n}$ be the tadpole graph obtained by joining a cycle
	$C_m$ to a path $P_n$ by a bridge. We prove that the independence
	polynomial of every tadpole graph is unimodal and establish
	sharp bounds for its mode. The unimodality
	result follows from a general criterion for graphs obtained
	by attaching a path to a fixed vertex. Over a field of
	characteristic zero, we also give a complete classification
	of the pairs $(m,n)$ for which the Artinian algebra defined
	by the edge ideal of $T_{m,n}$ together with the squares of
	all variables has the weak Lefschetz property.
\end{abstract}

\makeatletter
\@namedef{subjclassname@2020}{%
	\textup{2020} Mathematics Subject Classification}
\makeatother


\maketitle
\section{Introduction}
Let \(A=\bigoplus_{i=0}^{D}[A]_i\) be a standard graded Artinian algebra over a field \(\Bbbk\). We say that \(A\) has the \emph{weak Lefschetz property} (WLP) if there exists a linear form \(\ell\in[A]_1\) such that, for every \(i\), the multiplication map
\[
\times\ell:[A]_i\longrightarrow[A]_{i+1}
\]
is either injective or surjective. Furthermore, \(A\) is said to have the \emph{strong Lefschetz property} (SLP) if there exists a linear form \(\ell\) such that, for every pair of integers \(i\geq 0\) and \(j\geq 1\), the multiplication map
\[
\times\ell^j:[A]_i\longrightarrow[A]_{i+j}
\]
has maximal rank.

The study of Lefschetz properties has become a central topic in commutative algebra because of its rich interactions with many other areas of mathematics. Among these are the theory of partially ordered sets through the Dilworth number and representation theory via Schur--Weyl duality. We refer the reader to \cite{HMMNWW13,MN13} for surveys and further references.

The Lefschetz properties of Artinian \(\Bbbk\)-algebras arising from monomial ideals have been studied extensively; see, for example,
\cite{AB20,AN20,DN24,GLN22,H24,K25, MMO13,MM16,MNS20,MMN11,PT23}
and the references therein.

In the present paper, we restrict our attention to a particular family of Artinian algebras defined by quadratic monomial relations, which can be encoded by simple graphs. Let \(G=(V,E)\) be a simple graph with vertex set
\(V=\{1,2,\ldots,n\}\), and let
\[
R=\Bbbk[x_1,\ldots,x_n]
\]
be the standard graded polynomial ring over \(\Bbbk\). The \emph{edge ideal} of \(G\) is
\[
I(G)=\bigl(x_i x_j \mid \{i,j\}\in E\bigr)\subset R.
\]
Associated to \(G\), we consider the Artinian algebra
\[
A(G)=\frac{R}{(x_1^2,\ldots,x_n^2)+I(G)}.
\]
We shall refer to this algebra as the \emph{Artinian algebra associated to \(G\)}. This leads naturally to the following problem.

\begin{question}
	\label{quest_WLPofAG}
	For which graphs \(G\) does the algebra \(A(G)\) have the WLP or the SLP? When \(A(G)\) fails to have one of these Lefschetz properties, in which degrees do the corresponding multiplication maps fail to have maximal rank?
\end{question}
The Lefschetz properties of $A(G)$ have been investigated in \cite{HPS26,NT24b,T21} for several classes of graphs, including complete graphs, star graphs, barbell graphs, wheel graphs, paths, cycles, and lollipop graphs. Artinian algebras defined by quadratic monomial relations have also been considered in earlier and recent works by Michałek--Miró-Roig~\cite{MM16}, Migliore--Nagel--Schenck~\cite{MNS20}, Dao--Nair~\cite{DN24}, Kling~\cite{K24}, Holleben~\cite{H24,H25}, and Cooper et al.~\cite{CFHNVT24}.

In the present paper, we instead focus on tadpole graphs. Recall that for integers \(m\geq 3\) and \(n\geq 1\), the \emph{tadpole graph} \(T_{m,n}\) is obtained by joining a cycle graph \(C_m\) to a path graph \(P_n\) by a bridge. More precisely, we label the vertices of \(C_m\) by \(x_1,\ldots,x_m\) and the vertices of \(P_n\) by \(y_1,\ldots,y_n\), and add the bridge \(\{x_m,y_1\}\), see Figure~\ref{fig1}. 
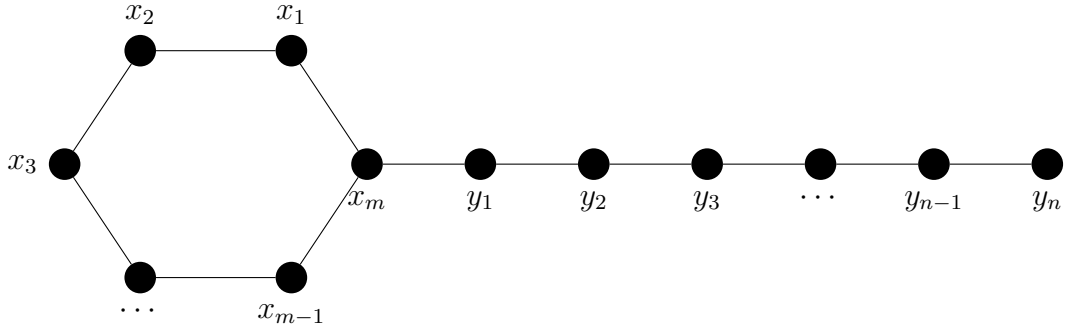
\begin{figure}[!ht]
	\begin{tikzpicture}[
		every edge/.style = {draw=black,very thick},
		vrtx/.style args = {#1/#2}{%
			circle, draw, thick, fill=black,
			minimum size=1mm, label=#1:#2}
		]
		\node (n1) [vrtx=above/$x_2$]  at (-1,0) {};
		\node (n2) [vrtx=above/$x_1$]at (1,0)  {};
		\node (n3) [vrtx=below/$x_m$]at (2,-1.5)  {};
		\node (n4) [vrtx=below/$x_{m-1}$]at (1,-3) {};
		\node (n5) [vrtx=below/$\cdots$]at (-1,-3)  {};
		\node (n6) [vrtx=left/$x_3$]at (-2,-1.5)  {};
		\node (n7) [vrtx=below/$y_1$]at (3.5,-1.5)  {};
		\node (n8) [vrtx=below/$y_2$]at (5,-1.5)  {};
		\node (n9) [vrtx=below/$y_3$]at (6.5,-1.5)  {};
		\node (n10) [vrtx=below/$\cdots$]at (8,-1.5)  {};
		\node (n11) [vrtx=below/$y_{n-1}$]at (9.5,-1.5)  {};
		\node (n12) [vrtx=below/$y_n$]at (11,-1.5)  {};
		\foreach \from/\to in {n1/n2,n2/n3,n3/n4,n4/n5,n5/n6,n6/n1, n3/n7, n8/n7, n8/n9, n9/n10, n10/n11, n11/n12}		
		\draw (\from) -- (\to);	
	\end{tikzpicture}
	\caption{Tadpole $T_{m,n}$}
	\label{fig1}
\end{figure}

Our main objective is to give a complete classification of the tadpole graphs
\(T_{m,n}\) for which the associated Artinian algebra \(A(T_{m,n})\) has, or fails
to have, the weak Lefschetz property. The classification is as follows.

\begin{theorem} \label{thm:mainTheorem_WLP}
Assume that $\operatorname{char}(\Bbbk)=0$.
The algebra \(A(T_{m,n})\) has the weak Lefschetz property if and only if one of the following conditions holds:
	\begin{enumerate}[\rm (i)]
		\item \(m\in\{3,6,10\}\) and \(n\in\{1,3,4,7\}\);
		\item \(m=4\) and \(n\in\{1,2,\ldots,7,9,10,13\}\);
		\item \(m\in\{5,8\}\) and \(n\in\{1,2,3,5,6,9\}\);
		\item \(m=7\) and \(n\in\{1,2,3,4,7\}\);
		\item $
		(m,n)\in
		\{(9,1),(9,4),(11,2),(11,3),(11,6),(12,1),
		(13,1),(13,4),(14,3),(16,1)\}.$
	\end{enumerate}
\end{theorem}

In addition to this classification, we study the independence polynomials 
of tadpole graphs. We prove that $I(T_{m,n};t)$ is unimodal for all 
$m\geq 3$ and $n\geq 1$, and we derive useful bounds for its mode. 
Recall that the independence polynomial of the path $P_n$ has mode
$\lambda_n=\left\lceil
\frac{5n+2-\sqrt{5n^2+20n+24}}{10}
\right\rceil$.
\begin{theorem}
The independence polynomial $I(T_{m,n};t)$ is unimodal for all $m\geq 3$ 
and $n\geq 1$. Moreover, for $m\geq 4$ and $n\geq 1$, its mode 
$\mu_{m,n}$ satisfies
\begin{equation*}
\lambda_{m-3}+\lambda_n-1
\leq \mu_{m,n}
\leq \lambda_{m-3}+\lambda_n+2.
\end{equation*}
\end{theorem}

The sharp mode bounds describe the location of the largest
coefficient of $I(T_{m,n};t)$ in terms of the modes of two
path independence polynomials. Both endpoints are attained,
as shown in Remark~\ref{rem:sharp-mode-bounds}.

Our approach to unimodality uses a general criterion for
graphs obtained by attaching a path to a fixed vertex.
The WLP classification uses the resulting unimodality
together with two algebraic obstructions: a failure of
surjectivity inherited from a path quotient and an explicit
nonzero element annihilated by the sum of the variables.
Together, these obstructions show that $A(T_{m,n})$ fails
the WLP whenever $m+n\geq18$. A coefficient comparison
handles $m+n=16$, and the remaining cases are settled by
exact computations in \texttt{Macaulay2}~\cite{HPS26M2}.

The paper is organized as follows. Section~2 recalls the necessary 
background on unimodal polynomials, independence polynomials, and the weak 
Lefschetz property. Section~3 establishes the unimodality of the independence 
polynomials of tadpole graphs and derives bounds for their modes. Finally, 
Section~4 proves the classification stated in Theorem~\ref{thm:mainTheorem_WLP}.

\section{Preliminaries}

This section reviews standard terminology and notation from commutative algebra and combinatorial commutative algebra, together with several auxiliary results that will be used in the subsequent sections.

\subsection{Unimodal polynomials and their modes}

\begin{definition}\label{mode}
Let $P(x)=\sum_{k=0}^{n}a_kx^k$ be a nonzero polynomial with nonnegative real coefficients.
	\begin{enumerate}[\quad \rm (i)]
		\item The polynomial $P(x)$ is said to be \emph{unimodal} if there exists an integer $m$ such that
		\[
		a_0 \leq a_1 \leq \cdots \leq a_{m-1} \leq a_m \geq a_{m+1} \geq \cdots \geq a_n.
		\]
		Set $a_{-1} = 0$. The \emph{mode} of the unimodal polynomial $\sum_{k=0}^{n} a_k x^k$ is defined to be the unique integer $i \in \{0, 1, \ldots, n\}$ satisfying
		\[
		a_{i-1} < a_i \geq a_{i+1} \geq \cdots \geq a_n.
		\]
		
		\item The polynomial $P(x)$ is said to be \emph{log-concave} if $a_k^2 \geq a_{k-1} a_{k+1}$ for all $1 \leq k \leq n-1$.
	\end{enumerate}
\end{definition}
For a nonnegative sequence with no internal zeros, log-concavity
implies unimodality. The following classical consequence of Newton's
inequalities gives a useful sufficient condition.
\begin{theorem}[{\rm \cite{PB15}}]
	If the polynomial $P(x) = a_0 + a_1x + \cdots + a_nx^n \in \mathbb{R}[x]$ has positive coefficients and all its roots are real, then it is log-concave and hence unimodal.
\end{theorem}

The determination of the mode relies heavily on the following result.

\begin{theorem}
[{\rm \cite[Théorème~B]{Be96}}]
\label[theorem]{thm_limit_mode}
	If the polynomial $P(x) = a_0 + a_1x + \cdots + a_nx^n \in \mathbb{R}[x]$ has positive coefficients and all its roots are real, then the mode $k$ of $P(x)$ satisfies
	\[
	\left\lfloor \frac{P'(1)}{P(1)} \right\rfloor \leq k \leq \left\lceil \frac{P'(1)}{P(1)} \right\rceil.
	\]
\end{theorem}

 We now study the modes of certain unimodal polynomials. Specifically, we provide some results that allow us to bound the modes and to show that the sum of two unimodal polynomials is also unimodal when their modes differ by at most one. Given a polynomial $f(x) = \sum_{i=0}^{n} a_i x^i$, we set $a_k = 0$ for all $k > n$ or $k < 0$.
 
 \begin{lemma}\cite[Lemma~4.1]{HPS26}\label[lemma]{unimodal_mode_unit}
 Let $f$ and $g$ be unimodal polynomials with nonnegative real coefficients, whose modes are $\mu$ and $\nu$, respectively. 
If $\mid\mu - \nu\mid \le 1$, then $f + g$ is also unimodal, and its mode belongs to $\{\min(\mu, \nu),\, \min(\mu, \nu) + 1\}$.
 \end{lemma}

\begin{lemma}\label{lem_compare_mode}
	Let $f$ and $g$ be two unimodal polynomials with nonnegative real coefficients such that $f + g$ is unimodal. 
	Denote by $p$, $q$, and $r$ the modes of $f$, $g$, and $f + g$, respectively. 
	Then 
	\[
	\min(p, q) \leq r \leq \max(p, q).
	\]
\end{lemma}

\begin{proof}
	Assume that
	\begin{align*}
		&f(x) = a_0 + a_1x + a_2x^2 + \cdots + a_{n-1}x^{n-1} + a_nx^n,\\
		&g(x) = b_0 + b_1x + b_2x^2 + \cdots + b_{m-1}x^{m-1} + b_mx^m.
	\end{align*}
	
	Then
	\begin{align*}
		&a_0 \leq a_1 \leq \cdots \leq a_{p-1} < a_p \geq a_{p+1} \geq \cdots \geq a_n,\\
		&b_0 \leq b_1 \leq \cdots \leq b_{q-1} < b_q \geq b_{q+1} \geq \cdots \geq b_m.
	\end{align*}
	
	\noindent\textsc{Claim 1:} $\min(p, q) \leq r$.
	
\noindent\textit{Proof of Claim 1:}	Assume to the contrary that $r < \min(p, q) = s$. 
	Then $r \leq s - 1 < s$. 
	Hence $a_{s-1} \leq a_s$ and $b_{s-1} \leq b_s$, which implies 
	$a_{s-1} + b_{s-1} \leq a_s + b_s$. 
	On the other hand, since $s - 1 \geq r$ and $r$ is the mode of $f + g$, we have 
	$a_{s-1} + b_{s-1} \geq a_s + b_s$. 
	Thus $a_{s-1} + b_{s-1} = a_s + b_s$, which forces 
	$a_s = a_{s-1}$ and $b_s = b_{s-1}$. 
	Since $s=\min\{p,q\}$, either $s=p$ or $s=q$. In the first case,
$a_{s-1}<a_s$, while in the second case, $b_{s-1}<b_s$.
Either conclusion contradicts the equalities obtained above.
	
	\noindent\textsc{Claim 2:} $r \leq \max(p, q)$.
	
\noindent\textit{Proof of Claim 2:}	Assume the contrary that $r > \max(p, q)$. 
	Then $r - 1 \geq \max(p, q)$, and hence 
	$a_{r-1} \geq a_r$ and $b_{r-1} \geq b_r$. 
	It follows that $a_{r-1} + b_{r-1} \geq a_r + b_r$. 
	However, since $r$ is the mode of $f + g$, we must have 
	$a_{r-1} + b_{r-1} < a_r + b_r$, a contradiction.
	
Claims 1 and 2 imply that $\min(p, q) \leq r \leq \max(p, q)$.
\end{proof}

\subsection{Graph theory}

From now on, by a \emph{graph} we mean a simple graph $G = (V, E)$ with vertex set $V = V(G)$ and edge set $E = E(G)$. 
We begin by recalling some basic definitions.

\begin{definition}
	Let $G = (V, E)$ be a graph.
	\begin{enumerate}[\quad \rm (i)]
		\item A subset $X$ of $V$ is called an \emph{independent set} of $G$ if for any $u, v \in X$, we have $\{u, v\} \notin E$; that is, the vertices in $X$ are pairwise nonadjacent. 
		If an independent set $X$ has $k$ elements, then $X$ is called an \emph{independent set of size $k$}, or a \emph{$k$-independent set} of $G$.

		\item The \emph{independence number} of a graph $G$ is the largest cardinality of an independent set of $G$. 
		We denote this number by $\alpha(G)$.
		
		\item  	The \emph{independence polynomial} of a graph $G$ is the polynomial in one variable $t$ whose coefficient of $t^k$ equals the number of independent sets of size $k$ in $G$. 
		We denote this polynomial by $I(G; t)$; that is,
		\[
		I(G; t) = \sum_{k=0}^{\alpha(G)} s_k(G) t^k,
		\]
		where $s_k(G)$ denotes the number of independent sets of size $k$ in $G$. 
		Note that $s_0(G) = 1$ since the empty set $\emptyset$ is an independent set of any graph $G$.
	\end{enumerate}
\end{definition}

The independence polynomial of a graph was introduced by Gutman and Harary in~\cite{GH83} as a generalization of the matching polynomial of a graph. 
For a vertex $v \in V$, define 
\[
N(v) = \{\, w \in V \mid \{v, w\} \in E \,\}
\quad \text{and} \quad
N[v] = N(v) \cup \{v\}.
\]
Some basic recursive formulas for computing the independence polynomial of a graph are recalled below (see, for instance, \cite{GH83, HL94}).

\begin{proposition}\label{FormulaforIndependence}
	Let $G_1$ and $G_2$ be vertex-disjoint graphs,
	and let $w$ be a vertex of a graph $G$.
	Then the following identities hold:
	\begin{enumerate}[\rm (i)]
		\item
		$I(G;t)=I(G\setminus w;t)
		+tI(G\setminus N[w];t)$;
		\item
		$I(G_1\cup G_2;t)=I(G_1;t)I(G_2;t)$.
	\end{enumerate}
\end{proposition}

An important result concerning the independence polynomials of graphs was proved by Chudnovsky and Seymour~\cite{CS07}.

\begin{theorem}
	The independence polynomial of a claw-free graph has only real zeros.
\end{theorem}

Recall that a graph is said to be \emph{claw-free} if it does not contain the complete bipartite graph $K_{1,3}$ as an induced subgraph.

\subsection{The weak Lefschetz property}

Let $R = \Bbbk[x_1, \ldots, x_n]$ be a polynomial ring over a field $\Bbbk$, where each variable $x_i$ has degree~$1$, and let $I \subset R$ be a homogeneous Artinian ideal. 
Then $A = R / I$ is a graded Artinian algebra, which can be written as
\[
A = \bigoplus_{i=0}^{D} [A]_i.
\]

Throughout this paper, we restrict our attention to Artinian algebras defined by monomial ideals. In this setting, it suffices to choose the Lefschetz element as the sum of all variables.
\begin{proposition}
[\cite{MMN11,LN19}]\label[proposition]{Proposition2.5}
Let $I \subset R=\Bbbk[x_1,\ldots,x_n]$ be an Artinian monomial ideal. Then the algebra $A=R/I$ has the WLP if and only if $\ell=x_1+x_2+\cdots+x_n$ is a Lefschetz element of $A$.
\end{proposition}

\Cref{Proposition2.5} was first proved in \cite[Proposition~2.2]{MMN11} over an infinite field, and later in \cite[Proposition~4.3]{LN19} over an arbitrary field.

Next, we present a necessary condition for the WLP of an Artinian graded algebra. 
First, we introduce some notation.

\begin{definition}
Let $A = \bigoplus_{j \geq 0} [A]_j$ be a standard graded $\Bbbk$-algebra. 
	The \emph{Hilbert series} of $A$ is the power series 
	\[
	HS(A, t) = \sum_{j \geq 0} \dim_\Bbbk [A]_j \, t^j.
	\]
	If $A$ is an Artinian graded algebra, then $[A]_i = 0$ for all $i \gg 0$. 
	We denote 
	\[
	D := \max \{\, i \mid [A]_i \neq 0 \,\},
	\]
	and call $D$ the \emph{socle degree} of $A$. 
	In this case, the Hilbert series of $A$ is a polynomial of the form
	\[
	HS(A, t) = 1 + h_1 t + \cdots + h_D t^D,
	\]
	where $h_i = \dim_\Bbbk [A]_i > 0$ for all $i$. 
	By definition, the degree of the Hilbert series of an Artinian graded algebra $A$ coincides with its socle degree $D = \max \{\, i \mid [A]_i \neq 0 \,\}$.
\end{definition}

\begin{proposition}[{\rm \cite[Proposition~3.2]{HMMNWW13}}]\label{Proposition2.8}
If $A$ has the WLP, then the Hilbert series of $A$ is unimodal.
\end{proposition}


\subsection{Artinian algebras associated to graphs}
We retain the notation $A(G)$ for the Artinian algebra associated to
a graph $G$, as introduced in Section~1. The following result relates
the combinatorial structure of $G$ to the Hilbert series of $A(G)$.

\begin{proposition}[{\rm \cite{NT24b}}]
The Hilbert series of $A(G)$ coincides with the independence polynomial of $G$; that is,
	\[
	HS(A(G), t) = I(G; t).
	\]
\end{proposition}
Nguyen and the first author in~\cite{NT24b} studied the independence polynomials of path graphs $P_n$ on $n$ vertices as well as their associated Artinian algebras. 
Recall that the independence
polynomial of the path $P_n$ has mode $\lambda_n$.

\begin{theorem}[{\rm \cite[Theorem~4.2]{NT24b}}]\label{WLP for P_n}
Suppose that $\operatorname{char}(\Bbbk) = 0$. Then the algebra $A(P_n)$ has the WLP if and only if $n \in \{1, 2, \dots, 7, 9, 10, 13\}.$
Furthermore, if $n\geq17$, then multiplication by the sum $\ell$ of
the variables is not surjective from degree $\lambda_n$ to degree
$\lambda_n+1$; that is,
\[
\times\ell:[A(P_n)]_{\lambda_n}
\longrightarrow[A(P_n)]_{\lambda_n+1}
\]
is not surjective.
\end{theorem}


\section{Independence polynomials of the tadpole graphs}

In this section, we prove that the independence polynomial of every
tadpole graph is unimodal and then bound its mode. We begin with a
tail construction for an arbitrary graph.
\begin{definition}
Let $G$ be a graph and let $v$ be a vertex of $G$. For each $n\geq1$,
the \emph{$n$-tail graph of $G$ at $v$}, denoted by
$\tail{G,v,n}$, is obtained by joining $v$ to an endpoint of a path
$P_n$ by a bridge; see Figure~\ref{fig2}.
\end{definition}

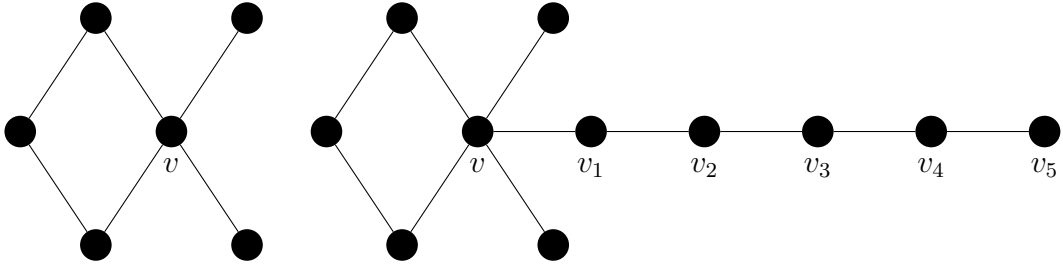
\begin{figure}[!ht]
\begin{tabular}{cc}
\begin{tikzpicture}[every edge/.style = {draw=black,very thick},vrtx/.style args = {#1/#2}{circle, draw, thick, fill=black,minimum size=1mm, label=#1:#2}]
\node (n1) [vrtx=above/]  at (-1,0) {};
\node (n2) [vrtx=above/]at (1,0)  {};
\node (n3) [vrtx=below/$v$]at (0,-1.5)  {};
\node (n4) [vrtx=below/]at (1,-3) {};
\node (n5) [vrtx=below/]at (-1,-3)  {};
\node (n6) [vrtx=left/]at (-2,-1.5)  {};
\foreach \from/\to in {n1/n3,n2/n3,n3/n4,n3/n5,n5/n6,n6/n1}		
\draw (\from) -- (\to);	
\end{tikzpicture} &
\begin{tikzpicture}[every edge/.style = {draw=black,very thick},vrtx/.style args = {#1/#2}{circle, draw, thick, fill=black,minimum size=1mm, label=#1:#2}]
\node (n1) [vrtx=above/]  at (-1,0) {};
\node (n2) [vrtx=above/]at (1,0)  {};
\node (n3) [vrtx=below/$v$]at (0,-1.5)  {};
\node (n4) [vrtx=below/]at (1,-3) {};
\node (n5) [vrtx=below/]at (-1,-3)  {};
\node (n6) [vrtx=left/]at (-2,-1.5)  {};
\node (n7) [vrtx=below/$v_1$]at (1.5,-1.5)  {};
\node (n8) [vrtx=below/$v_2$]at (3,-1.5)  {};
\node (n9) [vrtx=below/$v_3$]at (4.5,-1.5)  {};
\node (n10) [vrtx=below/$v_4$]at (6,-1.5)  {};
\node (n11) [vrtx=below/$v_5$]at (7.5,-1.5)  {};
\foreach \from/\to in {n1/n3,n2/n3,n3/n4,n3/n5,n5/n6,n6/n1, n7/n3, n8/n7, n8/n9, n9/n10, n10/n11}		
\draw (\from) -- (\to);	
\end{tikzpicture}
\end{tabular}
\caption{$G$ (left) and $\tail{G,v,5}$ (right)}
\label{fig2}
\end{figure}

The unimodality of the independence polynomial of the $n$-tail graph of $G$ may be shown by induction as follows. Whenever $I(\tail{G,v,n};t)$ is unimodal, denote its mode by
$\nu_n^v(G)$. The following result propagates unimodality along the
tail.

\begin{theorem}\label{unimodal_tail}
Let $G$ be a graph and let $v$ be a vertex of $G$. Assume that the independence polynomials $I(\tail{G,v,1};t)$ and $I(G;t)$ are unimodal with modes $\nu_1^v(G)$ and $\nu$, respectively, satisfying $0\le \nu_1^v(G)- \nu \leq 1$. Then $I(\tail{G,v,n};t)$ is unimodal for all $n\geq 1$. Moreover, $0\le \nu_n^v(G)- \nu_{n-1}^v(G)\leq 1$ for all $n\geq 2$. 
\end{theorem}
\begin{proof}
The consecutive vertices of $P_n$ are labeled as $v_1,v_2,\ldots,v_n$ such that $v$ and $v_1$ are connected by a bridge (see \Cref{fig2}).  We proceed by induction on $n\ge 2$.
 
If $n=2$, then by applying \Cref{FormulaforIndependence}(i) to the vertex numbered $v_2$
\begin{align*}
 I(\tail{G,v,2};t) &= I(\tail{G,v,2}\setminus v_2;t)+ tI(\tail{G,v,2}\setminus N[v_2];t)\\
  &= I(\tail{G,v,1};t) + tI(G;t).
 \end{align*}
If $\nu_1^v(G)= \nu$, then $I(\tail{G,v,2};t)$ is unimodal with mode $\nu_2^v(G)\in \brb{\nu,\nu+1}$ by \Cref{unimodal_mode_unit}; if $\nu_1^v(G)=\nu+1$, then it is immediate that $I(\tail{G,v,2};t)$ is unimodal with mode $\nu+1$. In both cases, we have $0\leq \nu_2^v(G)- \nu_1^v(G)\leq 1$. 

Assume that the assertion holds for every integer $k<n$, where $n\geq3$.
We prove it for $n$. Applying \Cref{FormulaforIndependence}(i) to the
vertex $v_n$, we obtain
\begin{align*}
I(\tail{G,v,n};t)
&=
I(\tail{G,v,n}\setminus v_n;t) +tI(\tail{G,v,n}\setminus N[v_n];t) \\
&=
I(\tail{G,v,n-1};t)
+tI(\tail{G,v,n-2};t).
\end{align*}
By the induction hypothesis, both
$I(\tail{G,v,n-1};t)$ and
$I(\tail{G,v,n-2};t)$ are unimodal and
\[
0\leq
\nu_{n-1}^v(G)-\nu_{n-2}^v(G)
\leq1.
\]
We distinguish the following two cases.

\noindent \textbf{Case 1: $\nu_{n-1}^v(G)=\nu_{n-2}^v(G)$.} In this case, the modes of
	$I(\tail{G,v,n-1};t)$ and
	$tI(\tail{G,v,n-2};t)$ are
	$\nu_{n-1}^v(G)$ and $\nu_{n-1}^v(G)+1$, respectively.
	By \Cref{unimodal_mode_unit},
	$I(\tail{G,v,n};t)$ is unimodal, and its mode satisfies
	\[
	\nu_n^v(G)\in
	\brb{\nu_{n-1}^v(G),\nu_{n-1}^v(G)+1}.
	\]
	Therefore,
	$0\leq \nu_n^v(G)-\nu_{n-1}^v(G)\leq 1$.

\noindent
\textbf{Case 2: $\nu_{n-1}^v(G)=\nu_{n-2}^v(G)+1$.}
The two summands in the recurrence have the same mode,
namely $\nu_{n-1}^v(G)$. Their sum is therefore unimodal
with mode
\[
\nu_n^v(G)=\nu_{n-1}^v(G).
\]
Thus $0\leq\nu_n^v(G)-\nu_{n-1}^v(G)\leq1$.

This completes the induction and hence the proof.
\end{proof}

Since $\tail{C_m,v,n}$ is exactly the tadpole graph $T_{m,n}$, we obtain the following.

\begin{corollary}\label{cor:unimodality_Tadpoles}
 The independence polynomial of $T_{m,n}$ is unimodal for all $m\geq 3$ and $n\geq 1$. 
\end{corollary}
\begin{proof}
	By~\cite[Proposition~3.3]{NT24b}, the independence
	polynomial of $C_m$ is unimodal with mode
	\[
	\rho_m=
	\left\lceil
	\frac{5m-4-\sqrt{5m^2-4}}{10}
	\right\rceil.
	\]
	For $m\geq5$, the mode estimates in
	\cite[Lemmas~3.4--3.6]{NT24b} imply that the
	independence polynomial of the pan graph $T_{m,1}$
	is unimodal with mode in $\{\rho_m,\rho_m+1\}$.
	
	For $m=3,4$, we have
	\[
	I(T_{3,1};t)=1+4t+2t^2,
	\qquad
	I(T_{4,1};t)=1+5t+5t^2+t^3.
	\]
	Both polynomials have mode $1$, and
	$\rho_3=\rho_4=1$. Thus the hypotheses of
	Theorem~\ref{unimodal_tail} hold for every $m\geq3$.
	Since $\tail{C_m,v,n}=T_{m,n}$, the result follows.
\end{proof}

To study the mode of the independence polynomial of $T_{m,n}$, we need the following. Recall that the mode of the independence polynomial of $P_n$ is $\lambda_n=\ceil{\frac{5n+2-\sqrt{5n^2+20n+24}}{10}}$, see~\cite[Proposition~3.1]{NT24b}.
\begin{lemma}\label{inequality_lambda}
For all $m,n\geq 1$, we have $\abs{\lambda_{m+n+2}-\lambda_m-\lambda_n}\leq 1$. 
\end{lemma}
\begin{proof}
It suffices to prove the following two claims.

\noindent\begin{claim}\label{claim 1}
    $\lambda_{m+n+2}\leq \lambda_m+\lambda_n+1$.
\end{claim}

\noindent\emph{Proof of Claim 1:}
For $d\geq 1$, set $A(d)= \sqrt{5d^2 + 20d+24}$. It suffices to show that
\begin{align*}
    1 + \frac{5m+2- A(m)}{10} + \frac{5n+2-A(n)}{10}\geq \frac{5(m+n+2)+2-A(m+n+2)}{10}.
\end{align*}
This is equivalent to showing that
\begin{align}
    2 + A(m+n+2)\geq A(m) + A(n). \label{1}
\end{align}

Since both sides of \eqref{1} are nonnegative, we may square them. After simplification, we obtain the equivalent inequality
\begin{align}
    2A(m+n+2)+5mn+10(m+n+2)
    \geq A(m)A(n).
    \label{2}
\end{align}

		Squaring both sides of (\ref{2}) and simplifying, we obtain the equivalent inequality
        \begin{align*}
            20(2mn + 4m+ 4n) + 160 + 2A(m+n+2)(10mn + 20m +20n +40)\geq 0.
        \end{align*}
		This inequality holds for all $m\geq 1$ and $n\geq 1$. Taking ceilings and using
$$\lceil x+y\rceil\leq\lceil x\rceil+\lceil y\rceil,$$ we obtain
\[
\lambda_{m+n+2}
\leq \lambda_m+\lambda_n+1,
\]
which completes the proof of Claim~\ref{claim 1}.
        
        Next, we have the following.

\noindent\begin{claim}\label{claim 2}
    $\lambda_m+\lambda_n-1\leq \lambda_{m+n+2}$.
\end{claim}

\noindent\emph{Proof of Claim 2:}		
Retaining the notation $A(d)= \sqrt{5d^2+20d+24}$, we first show that
\begin{align}
    5(m+n)+4- A(m) - A(n)\leq 5(m+n+2)+2 - A(m+n+2) \label{5}
\end{align}
This is equivalent to showing that
\begin{align}
     A(m+n+2)\leq 8 + A(m) + A(n). \label{3}
\end{align}

		Squaring both sides of (\ref{3}) and simplifying, we obtain the equivalent inequality
        \begin{align}
            5mn + 10(m+n)\leq 14 + A(m)A(n) + 8A(m) + 8A(n). \label{4}
        \end{align}
        Since $A(m)A(n) > 5mn$ and $8A(m)+ 8A(n)> 10(m+n)$ for all $m\geq 1$ and $n\geq 1$, inequality (\ref{4}) holds immediately, which in turn yields (\ref{5}).

        Applying the ceiling inequality $\ceil{x}+\ceil{y}-1\leq \ceil{x+y}$ along with (\ref{5}), we obtain
		\begin{align*}
			\lambda_m+\lambda_n-1&= \ceil{\frac{5m+2-A(m)}{10}}+\ceil{\frac{5n+2-A(n)}{10}}-1\\
			&\leq \ceil{\frac{5m+5n+4-A(m)-A(n)}{10}}\\
			&\leq \ceil{\frac{5(m+n+2)+2-A(m+n+2)}{10}} \\
			&=\lambda_{m+n+2},
		\end{align*}
        which completes the proof of Claim~\ref{claim 2}.

The result now follows from Claims~\ref{claim 1} and \ref{claim 2}.
\end{proof}

\begin{lemma}\label{lem:path_mean_mode}
For every integer $r\geq 1$, one has
$\dfrac{I'(P_r;1)}{I(P_r;1)}\leq \lambda_r+\dfrac12$.
\end{lemma}

\begin{proof}
Let $F_0=0$, $F_1=1$, and $F_{r+2}=F_{r+1}+F_r$ be the Fibonacci
numbers. Since
\begin{equation*}
I(P_r;t)=\sum_{j\geq0}\binom{r-j+1}{j}t^j,
\end{equation*}
we have $I(P_r;1)=F_{r+2}$. Combinatorially, the quantity $I'(P_r;1)$ counts pairs $(S,v)$, where
$S$ is an independent set of $P_r$ and $v\in S$ is a distinguished vertex.

Write $V(P_r)=\{v_1,\ldots,v_r\}$, where $v_i$ is adjacent to 
$v_{i+1}$. Fix a distinguished vertex $v_j$. After choosing $v_j$,
the vertices that may still be selected lie to the left
of $v_{j-1}$ or to the right of $v_{j+1}$.
The numbers of possible independent sets on these two
sides are $F_j$ and $F_{r-j+1}$, respectively.
At an endpoint, the corresponding side is empty and
contributes one choice, consistently with $F_1=F_2=1$.
Thus the number of independent sets containing $v_j$
is $F_jF_{r-j+1}$. Summing over all possible choices of $v_j$, we obtain
\begin{equation*}
I'(P_r;1)
=
\sum_{j=1}^{r} F_jF_{r-j+1}.
\end{equation*}
Using the Fibonacci convolution identity, we obtain
\begin{equation*}
I'(P_r;1)=\frac{rF_{r+2}+(r+2)F_r}{5}.
\end{equation*}
Consequently,
\[
\frac{I'(P_r;1)}{I(P_r;1)}
=
\frac15\left(r+(r+2)\frac{F_r}{F_{r+2}}\right).
\]
To estimate the last quotient, let
$\varphi=(1+\sqrt5)/2$ and
$\alpha=\varphi^{-2}=(3-\sqrt5)/2$. By Binet's formula,
$F_r=(\varphi^r-\psi^r)/\sqrt5$, where
$\psi=(1-\sqrt5)/2$. Since $\psi/\varphi=-\alpha$, we obtain
\[
\frac{F_r}{F_{r+2}}
=
\alpha\,
\frac{1-(-\alpha)^r}{1-(-\alpha)^{r+2}}.
\]
Set
\[
B_r:=
\frac{3r+7-\sqrt{5r^2+20r+24}}{2(r+2)}.
\]
We show that $F_r/F_{r+2}\leq B_r$.

First, observe that
\[
B_r-\alpha
=
\frac{
1+\sqrt5(r+2)-\sqrt{5(r+2)^2+4}
}{
2(r+2)
}>0.
\]
Indeed, $\sqrt{5(r+2)^2+4}<1+\sqrt5(r+2)$ for every $r\geq1$.

Suppose first that $r$ is even. Then
\[
\frac{F_r}{F_{r+2}}
=
\alpha\frac{1-\alpha^r}{1-\alpha^{r+2}}
<\alpha<B_r.
\]

Now suppose that $r$ is odd. For $r=1$, one checks directly that
$F_1/F_3=B_1=1/2$. Assume that $r\geq3$. Then
\begin{align*}
\frac{F_r}{F_{r+2}}-\alpha
&=
\alpha
\left(
\frac{1+\alpha^r}{1+\alpha^{r+2}}-1
\right)\\
&=
\frac{\alpha^{r+1}(1-\alpha^2)}
     {1+\alpha^{r+2}}
<\alpha^{r+1}.
\end{align*}
Since $\alpha<2/5$, an elementary induction on odd $r\geq3$ gives
$\alpha^{r+1}\leq 1/\bigl(5(r+2)\bigr)$.

On the other hand, since
$\sqrt5(r+2)\geq3\sqrt5>91/30$, we have
\[
\sqrt{5(r+2)^2+4}
\leq
\sqrt5(r+2)+\frac35.
\]
Therefore,
\[
B_r-\alpha
\geq
\frac{1}{5(r+2)}.
\]
It follows that
\[
\frac{F_r}{F_{r+2}}-\alpha
<
\alpha^{r+1}
\leq
\frac{1}{5(r+2)}
\leq
B_r-\alpha.
\]
Hence $F_r/F_{r+2}\leq B_r$ also when $r$ is odd. We conclude that
\[
\frac{F_r}{F_{r+2}}
\leq
\frac{3r+7-\sqrt{5r^2+20r+24}}{2(r+2)}
\]
for every $r\geq1$.
It follows that
\[
\frac{I'(P_r;1)}{I(P_r;1)}
\leq
\frac{5r+7-\sqrt{5r^2+20r+24}}{10}.
\]
Recall that
$\lambda_r=\left\lceil
\bigl(5r+2-\sqrt{5r^2+20r+24}\bigr)/10
\right\rceil$. Therefore,
\[
\frac{I'(P_r;1)}{I(P_r;1)}
\leq
\frac{5r+2-\sqrt{5r^2+20r+24}}{10}
+\frac12
\leq
\lambda_r+\frac12.
\]
This completes the proof.
\end{proof}

\begin{theorem}\label{thm:mode_tadpole}
Let $\mu_{m,n}$ denote the mode of the independence polynomial of
$T_{m,n}$. Then
$\lambda_{m-3}+\lambda_n-1\leq\mu_{m,n}
\leq\lambda_{m-3}+\lambda_n+2$
for all $m\geq4$ and $n\geq1$.
\end{theorem}

\begin{proof}
Using the labeling in Figure~\ref{fig1}, deleting $x_{m-1}$ leaves
a path on $m+n-1$ vertices, whereas deleting $N[x_{m-1}]$ leaves
the disjoint union $P_{m-3}\cup P_n$. Hence Proposition~\ref{FormulaforIndependence}
gives 
\begin{align*}
I(T_{m,n};t)
&=I(T_{m,n}\setminus x_{m-1};t)
+tI(T_{m,n}\setminus N[x_{m-1}];t)\\
&=I(P_{m+n-1};t)
+tI(P_{m-3};t)I(P_n;t).
\end{align*}

Set $f(t):=I(P_{m-3};t)I(P_n;t)$ and
$s:=\lambda_{m-3}+\lambda_n$.
Since $P_{m-3}\cup P_n$ is claw-free, the polynomial $f(t)$ has only
real negative roots. In particular, $f(t)$ is log-concave and hence
unimodal. By Theorem \ref{thm_limit_mode}, if $\mu$ denotes the mode of $f$, then
    \begin{align*}
        \flo{\frac{f'(1)}{f(1)}}\leq \mu\leq \ceil{\frac{f'(1)}{f(1)}}.
    \end{align*}
On the other hand,
\[
\frac{f'(1)}{f(1)}
=
\frac{I'(P_{m-3};1)}{I(P_{m-3};1)}
+
\frac{I'(P_n;1)}{I(P_n;1)}.
\]
Applying Lemma~\ref{lem:path_mean_mode} to $P_{m-3}$ and $P_n$, we
obtain
\[
\frac{f'(1)}{f(1)}
\leq
\lambda_{m-3}+\frac12+\lambda_n+\frac12
=s+1.
\]
Hence $\mu\leq
\left\lceil f'(1)/f(1)\right\rceil\leq s+1$. Therefore, the mode of
$tf(t)$ is at most $s+2$.

Moreover, Lemma~\ref{inequality_lambda} gives
$|\lambda_{m+n-1}-\lambda_{m-3}-\lambda_n|\leq1$, and hence
$\lambda_{m+n-1}\leq s+1$. Since $I(T_{m,n};t)$ is unimodal, Lemma~\ref{lem_compare_mode} yields
\[
\mu_{m,n}
\leq
\max\{\lambda_{m+n-1},1+\mu\}
\leq s+2.
\]
Applying Theorem~\ref{thm_limit_mode} to each path
independence polynomial, and using the elementary
floor and ceiling inequalities, we obtain
\begin{align*}
        \flo{\frac{f'(1)}{f(1)}}&\geq \flo{\frac{I'(P_{m-3};1)}{I(P_{m-3};1)}}+\flo{\frac{I'(P_n;1)}{I(P_n;1)}}\\
        &\geq \ceil{\frac{I'(P_{m-3};1)}{I(P_{m-3};1)}}+ \ceil{\frac{I'(P_n;1)}{I(P_n;1)}}-2\\
        &\geq \lambda_{m-3}+\lambda_n-2.
    \end{align*}
    Consequently, the mode of $tf(t)$ is $\mu+1$, which satisfies
    \begin{align*}
        \lambda_{m-3}+\lambda_n-1\leq \mu+1\leq \lambda_{m-3}+\lambda_n+2.
    \end{align*}
Applying Lemmas~\ref{lem_compare_mode} and~\ref{inequality_lambda}, we obtain
$\abs{\lambda_{m+n-1}-\lambda_{m-3}-\lambda_n}\leq 1$ and
\[
\min\brb{\lambda_{m+n-1},\mu+1}
\leq \mu_{m,n}
\leq
\max\brb{\lambda_{m+n-1},\mu+1}.
\]
Combining these inequalities with the bounds for $\mu+1$ obtained above, we conclude that
$\lambda_{m-3}+\lambda_n-1\leq\mu_{m,n}
\leq\lambda_{m-3}+\lambda_n+2$.
\end{proof}
\begin{remark}\label{rem:sharp-mode-bounds}
Both bounds in Theorem~\ref{thm:mode_tadpole} are sharp. The upper bound is attained by $T_{4,19}$. Indeed,
$\lambda_1=0$ and $\lambda_{19}=5$, while the maximum coefficient of
$I(T_{4,19};t)$ occurs in degree $7$. Hence,
$\mu_{4,19}=7=\lambda_1+\lambda_{19}+2$.

For the lower bound, consider $T_{221,218}$. We have
$\lambda_{218}=61$. If $h_k$ is the coefficient of $t^k$ in $I(T_{221,218};t)$, then a direct computation using
\[
h_k=\binom{439-k}{k}
+\sum_{i+j=k-1}
\binom{219-i}{i}\binom{219-j}{j}
\]
gives $h_{121}>h_{122}$. Since $I(T_{221,218};t)$ is unimodal by Corollary~\ref{cor:unimodality_Tadpoles}, and $\mu_{221,218}\ge 2\lambda_{218}-1$, it
follows that
$\mu_{221,218}=121=2\lambda_{218}-1$. Thus, the lower bound is also
attained.
\end{remark}

\section{The WLP for tadpole graphs}
Throughout this section, $\Bbbk$ denotes a field of characteristic
zero. We use the labeling of $T_{m,n}$ shown in
Figure~\ref{fig1}. 
Let $A(T_{m,n})$
be the Artinian algebra associated to the tadpole graph.
In this section, we prove Theorem~\ref{thm:mainTheorem_WLP}. We begin by
recalling the known classifications for the cases $m=3$ and $n=1$.
The algebras $A(T_{3,n})$ and $A(T_{m,1})$ were studied by Nguyen and
the first author in~\cite{NT24b}. Their results are summarized as follows.

\begin{proposition}[{\cite[Corollary~4.5 and Theorem~4.7]{NT24b}}]
\label{pro:WLP_small_m_and_n}
The following statements hold:
\begin{itemize}
\item[(a)] $A(T_{3,n})$ has the WLP if and only if $n\in\{1,3,4,7\}$.
\item[(b)] $A(T_{m,1})$ has the WLP if and only if
$m\in\{3,4,\ldots,10,12,13,16\}$.
\end{itemize}
\end{proposition}
In view of Proposition~\ref{pro:WLP_small_m_and_n}, it remains to
consider the range $m\geq4$ and $n\geq2$. We first combine a path
quotient with an explicit kernel element to show that
$A(T_{m,n})$ fails the WLP whenever $m+n\geq18$.
The case $m+n=16$ is then excluded by comparing two consecutive
coefficients of the Hilbert series and applying the same kernel
construction. Finally, the remaining cases, with $m+n\leq15$
or $m+n=17$, are settled by exact computations over $\mathbb{Q}$
in \texttt{Macaulay2}.

We begin with a lemma that transfers failures of surjectivity
from graded quotients to the original algebra.

\begin{lemma}\label{lem:failure_transfer}
Let $A=R/I$ be an Artinian monomial algebra over a field $\Bbbk$ of
characteristic zero. Assume that the Hilbert function of $A$ is
unimodal with mode $\mu$, and let $\ell_A$ denote the sum of the images
in $A$ of the variables of $R$.

Let $\pi:A\twoheadrightarrow B$ be a graded surjection, and set
$\ell_B=\pi(\ell_A)$. If multiplication by $\ell_B$ is not surjective from $[B]_d$ to
$[B]_{d+1}$ for some $d\geq\mu$, then multiplication by $\ell_A$ from
$[A]_d$ to $[A]_{d+1}$ is not surjective. In particular, $A$ does not
have the WLP.
\end{lemma}

\begin{proof}
Since $I$ is monomial, Proposition~\ref{Proposition2.5} shows that $A$
has the WLP if and only if $\ell_A$ is a Lefschetz element of $A$.
Consider the commutative diagram
\begin{equation*}
\begin{CD}
[A]_d @>{\times\ell_A}>> [A]_{d+1}\\
@V{\pi}VV @VV{\pi}V\\
[B]_d @>{\times\ell_B}>> [B]_{d+1}.
\end{CD}
\end{equation*}
If the upper horizontal map were surjective, then the lower horizontal
map would also be surjective, contrary to the assumption. Since
$d\geq\mu$ and the Hilbert function of $A$ is unimodal, we have
$\dim_{\Bbbk}[A]_d\geq\dim_{\Bbbk}[A]_{d+1}$. Thus, maximal rank in
this degree requires surjectivity. Therefore, $\ell_A$ is not a
Lefschetz element of $A$, and hence $A$ does not have the WLP.
\end{proof}

The preceding lemma detects failures of the WLP inherited from graded
quotients. We shall also need a direct obstruction to injectivity in
$A(T_{m,n})$. By Proposition~\ref{Proposition2.5}, it suffices to
consider the linear form
$\ell=x_1+\cdots+x_m+y_1+\cdots+y_n$.
The following lemma constructs a nonzero element of $A(T_{m,n})$
annihilated by $\ell$ in degree $\lceil(m+n)/4\rceil$.

\begin{lemma}\label{lem:block-annihilator}
Set $N=m+n$ and $d=\lceil N/4\rceil$. Then the multiplication map
\begin{equation*}
\times\ell:[A(T_{m,n})]_d\longrightarrow[A(T_{m,n})]_{d+1}
\end{equation*}
is not injective. Consequently, if
$\dim_{\Bbbk}[A(T_{m,n})]_d\leq
\dim_{\Bbbk}[A(T_{m,n})]_{d+1}$, then $A(T_{m,n})$ does not have the
WLP.
\end{lemma}

\begin{proof}
Write $N=4q+r$, where $0\leq r\leq3$. Order the vertices by setting
$z_i=x_i$ for $1\leq i\leq m$ and $z_{m+j}=y_j$ for
$1\leq j\leq n$. Then $z_1,\ldots,z_N$ form a spanning path of
$T_{m,n}$, and the only edge of $T_{m,n}$ not belonging to this path
is $\{z_1,z_m\}$.

For each $0\leq j\leq q-1$, set
$a_j=z_{4j+1}+z_{4j+2}$,
$b_j=z_{4j+3}+z_{4j+4}$, and $F_j=a_j-b_j$.
Since consecutive vertices along the spanning path are adjacent and
the square of every variable is zero in $A(T_{m,n})$, we have
$a_j^2=b_j^2=0$. Hence,
\begin{equation*}
(a_j+b_j)F_j=(a_j+b_j)(a_j-b_j)=a_j^2-b_j^2=0.
\end{equation*}

If $r=0$, define
$F=\prod_{j=0}^{q-1}F_j$.

Suppose that $r>0$, and set
$H=z_{4q+1}+\cdots+z_{4q+r}$. Define
\begin{equation*}
G=
\begin{cases}
z_{4q+1},&r=1,\\
z_{4q+1}-z_{4q+2},&r=2,\\
z_{4q+1}-z_{4q+3},&r=3.
\end{cases}
\end{equation*}
Then $HG=0$. Indeed, this follows from the square-zero relations when
$r=1,2$. When $r=3$, we have
\begin{equation*}
HG
=z_{4q+2}z_{4q+1}-z_{4q+2}z_{4q+3}
=0,
\end{equation*}
because $z_{4q+2}$ is adjacent to both $z_{4q+1}$ and $z_{4q+3}$.

For every $r>0$, define
$F=\bigl(\prod_{j=0}^{q-1}F_j\bigr)G$.
The linear form $\ell$ can be written as
$$\ell=\sum_{j=0}^{q-1}(a_j+b_j)$$ when $r=0$, and as
$$\ell=\sum_{j=0}^{q-1}(a_j+b_j)+H$$ when $r>0$. Since
$(a_j+b_j)F_j=0$ for every $j$ and $HG=0$ when $r>0$, it follows that
$\ell F=0$. Moreover, $\deg F=q$ when $r=0$ and $\deg F=q+1$ when
$r>0$. Therefore, $F\in[A(T_{m,n})]_d$.

It remains to show that $F\neq0$. Set
$M_0=\prod_{j=0}^{q-1}z_{4j+2}$ and define
\begin{equation*}
M=
\begin{cases}
M_0,&r=0,\\
M_0z_{4q+1},&r>0.
\end{cases}
\end{equation*}
The monomial $M$ occurs in $F$ with coefficient $1$: it is obtained
uniquely by selecting $z_{4j+2}$ from each factor $F_j$ and, when
$r>0$, selecting $z_{4q+1}$ from $G$.

The indices of the variables occurring in $M_0$ differ successively
by $4$. When $r>0$, the last two selected indices are $4q-2$ and
$4q+1$, which differ by $3$. Thus, no two variables in $M$ correspond
to adjacent vertices of the spanning path. Furthermore, $M$ does not
contain $z_1$, so its support does not contain the additional edge
$\{z_1,z_m\}$. Hence, the support of $M$ is an independent set of
$T_{m,n}$. Therefore, $M$ is nonzero in $A(T_{m,n})$, and consequently
$F\neq0$.

Thus, $F$ is a nonzero element of $[A(T_{m,n})]_d$ satisfying
$\ell F=0$. Hence, multiplication by $\ell$ from degree $d$ to degree
$d+1$ is not injective.

Finally, if
$\dim_{\Bbbk}[A(T_{m,n})]_d\leq
\dim_{\Bbbk}[A(T_{m,n})]_{d+1}$, then maximal rank in this degree
requires injectivity. Therefore, $A(T_{m,n})$ does not have the WLP.
\end{proof}
\begin{lemma}\label{lower bound for lambda}
	For every integer $N\geq18$, one has
	\[
	\left\lceil \frac{N}{4}\right\rceil\leq\lambda_{N-1}.
	\]
\end{lemma}

\begin{proof}
	For $N=18,19$, the assertion follows from $\lambda_{17}=\lambda_{18}=5.$
	
	Suppose that $N\geq20$. By the formula for the mode of a path,
	\[
	\lambda_{N-1}
	=
	\left\lceil
	\frac{5N-3-\sqrt{5N^2+10N+9}}{10}
	\right\rceil.
	\]
	We claim that
	\[
	\frac{N}{4}
	\leq
	\frac{5N-3-\sqrt{5N^2+10N+9}}{10}.
	\]
	Indeed, this inequality is equivalent to
	\[
	5N-6\geq2\sqrt{5N^2+10N+9}.
	\]
	Since both sides are nonnegative, squaring shows that it is
	equivalent to
	\[
	5N(N-20)\geq0,
	\]
	which holds for $N\geq20$. Taking ceilings gives the result.
\end{proof}

\begin{proposition}\label{pro:large-total-size}
	Assume that $\operatorname{char}(\Bbbk)=0$.
	Let $m\geq4$ and $n\geq2$. If $m+n\geq18$, then
	$A(T_{m,n})$ does not have the weak Lefschetz property.
\end{proposition}

\begin{proof}
	Set
	\[
	A=A(T_{m,n}),\qquad
	N=m+n,\qquad
	d=\left\lceil\frac{N}{4}\right\rceil,
	\qquad
	p=\lambda_{N-1},
	\]
	and let $\ell$ be the sum of all the variables.
	By Corollary~\ref{cor:unimodality_Tadpoles}, the Hilbert
	function of $A$ is unimodal. Denote its mode by $\mu$.
	Lemma~\ref{lower bound for lambda} gives $d\leq p$.
	
	Suppose first that $\mu\leq p$. Deleting $x_{m-1}$ from
	$T_{m,n}$ leaves a path on $N-1$ vertices. Hence there is
	a graded surjection
	\[
	A\twoheadrightarrow A/(x_{m-1})\cong A(P_{N-1}).
	\]
	Since $N-1\geq17$, Theorem~\ref{WLP for P_n} shows that
	multiplication by the image of $\ell$ on $A(P_{N-1})$
	is not surjective from degree $p$ to degree $p+1$.
	As $p\geq\mu$, Lemma~\ref{lem:failure_transfer} implies
	that $A$ does not have the WLP.
	
	Suppose next that $\mu>p$. Then $d\leq p<\mu.$
	The unimodality of the Hilbert function yields
	\[
	\dim_{\Bbbk}[A]_d\leq\dim_{\Bbbk}[A]_{d+1}.
	\]
	On the other hand, Lemma~\ref{lem:block-annihilator}
	shows that multiplication by $\ell$ from degree $d$
	to degree $d+1$ is not injective. Therefore, this map
	does not have maximal rank, and $A$ does not have
	the WLP.
\end{proof}

We next consider the case $m+n=16$. The following
coefficient comparison allows us to apply
Lemma~\ref{lem:block-annihilator}.

\begin{lemma}\label{lem:finite-mode-comparison}
	Let $m\geq4$ and $n\geq2$ satisfy $m+n=16$, and write
	\[
	I(T_{m,n};t)=\sum_{k\geq0}h_k(m,n)t^k.
	\]
	Then
	\[
	h_4(m,n)<h_5(m,n).
	\]
\end{lemma}

\begin{proof}
Applying Proposition~\ref{FormulaforIndependence}(i) to the vertex $x_{m-1}$, we obtain
	\[
	I(T_{m,n};t)
	=
	I(P_{15};t)+tI(P_{m-3};t)I(P_n;t).
	\]
By \cite[Proposition~3.1]{NT24b}, one has
	\[
	I(P_r;t)
	=
	\sum_{j=0}^{\lfloor(r+1)/2\rfloor}
	\binom{r-j+1}{j}t^j.
	\]
It follows that
	\[
	h_k(m,n)
	=
	\binom{16-k}{k}
	+
	\sum_{\substack{i,j\geq0\\i+j=k-1}}
	\binom{m-2-i}{i}\binom{n+1-j}{j}.
	\]
We use the convention that $\binom{a}{b}=0$
unless $0\leq b\leq a$.
	
	Since $n=16-m$, the assumptions give $4\leq m\leq14$.
	Evaluating the preceding formula for $k=4,5$ yields
	\[
	\begin{array}{c|c|c}
		m & h_4(m,16-m) & h_5(m,16-m)\\
		\hline
		4 & 670 & 708\\
		5,\ 14 & 669 & 700\\
		6\leq m\leq13 & 669 & 701
	\end{array}
	\]
	Thus $h_4(m,n)<h_5(m,n)$ in every case.
\end{proof}

\begin{proposition}\label{lem:special-sums}
	Assume that $\operatorname{char}(\Bbbk)=0$.
	Let $m\geq4$ and $n\geq2$ satisfy $m+n=16$.
	Then $A(T_{m,n})$ does not have the WLP.
\end{proposition}

\begin{proof}
	Set $A=A(T_{m,n})$, and let $\ell$ be the sum of all
	the variables. Since the Hilbert series of $A$ equals
	$I(T_{m,n};t)$, Lemma~\ref{lem:finite-mode-comparison}
	gives
	\[
	\dim_{\Bbbk}[A]_4<\dim_{\Bbbk}[A]_5.
	\]
	On the other hand,
	\[
	\left\lceil\frac{m+n}{4}\right\rceil=4.
	\]
	Therefore, Lemma~\ref{lem:block-annihilator} shows that
	\[
	\times\ell:[A]_4\longrightarrow[A]_5
	\]
	is not injective. Hence this map does not have maximal
	rank. By Proposition~\ref{Proposition2.5}, the algebra
	$A$ does not have the WLP.
\end{proof}

\begin{table}[!ht]
	\centering
	\caption{The weak Lefschetz property of $A(T_{m,n})$}
	\label{tab:Tadpole_WLP}
	
	\newcommand{\wlpYes}{\textcolor{blue}{\ensuremath{\checkmark}}}
	\newcommand{\wlpNo}{\textcolor{blue}{\ensuremath{\times}}}
	\newcommand{\wlpLarge}{%
		\textcolor{black}{\ensuremath{\times}}}
		\newcommand{\wlpSixteen}{%
			\textcolor{red}{\ensuremath{\times}}}
	
	\begingroup
	\setlength{\tabcolsep}{5pt}
	\renewcommand{\arraystretch}{1.15}
	\begin{tabular}{|c|*{15}{c|}}
		\hline
		\diagbox{$m$}{$n$}
		& 2 & 3 & 4 & 5 & 6 & 7 & 8 & 9
		& 10 & 11 & 12 & 13 & 14 & 15 & 16\\
		\hline
		4
		& \wlpYes & \wlpYes & \wlpYes & \wlpYes & \wlpYes
		& \wlpYes & \wlpNo & \wlpYes & \wlpYes & \wlpNo
		& \wlpSixteen & \wlpYes & \wlpLarge & \wlpLarge & \wlpLarge\\
		\hline
		5
		& \wlpYes & \wlpYes & \wlpNo & \wlpYes & \wlpYes
		& \wlpNo & \wlpNo & \wlpYes & \wlpNo & \wlpSixteen
		& \wlpNo & \wlpLarge & \wlpLarge & \wlpLarge & \wlpLarge\\
		\hline
		6
		& \wlpNo & \wlpYes & \wlpYes & \wlpNo & \wlpNo
		& \wlpYes & \wlpNo & \wlpNo & \wlpSixteen & \wlpNo
		& \wlpLarge & \wlpLarge & \wlpLarge & \wlpLarge & \wlpLarge\\
		\hline
		7
		& \wlpYes & \wlpYes & \wlpYes & \wlpNo & \wlpNo
		& \wlpYes & \wlpNo & \wlpSixteen & \wlpNo & \wlpLarge
		& \wlpLarge & \wlpLarge & \wlpLarge & \wlpLarge & \wlpLarge\\
		\hline
		8
		& \wlpYes & \wlpYes & \wlpNo & \wlpYes & \wlpYes
		& \wlpNo & \wlpSixteen & \wlpYes & \wlpLarge & \wlpLarge
		& \wlpLarge & \wlpLarge & \wlpLarge & \wlpLarge & \wlpLarge\\
		\hline
		9
		& \wlpNo & \wlpNo & \wlpYes & \wlpNo & \wlpNo
		& \wlpSixteen & \wlpNo & \wlpLarge & \wlpLarge & \wlpLarge
		& \wlpLarge & \wlpLarge & \wlpLarge & \wlpLarge & \wlpLarge\\
		\hline
		10
		& \wlpNo & \wlpYes & \wlpYes & \wlpNo & \wlpSixteen
		& \wlpYes & \wlpLarge & \wlpLarge & \wlpLarge & \wlpLarge
		& \wlpLarge & \wlpLarge & \wlpLarge & \wlpLarge & \wlpLarge\\
		\hline
		11
		& \wlpYes & \wlpYes & \wlpNo & \wlpSixteen & \wlpYes
		& \wlpLarge & \wlpLarge & \wlpLarge & \wlpLarge & \wlpLarge
		& \wlpLarge & \wlpLarge & \wlpLarge & \wlpLarge & \wlpLarge\\
		\hline
		12
		& \wlpNo & \wlpNo & \wlpSixteen & \wlpNo & \wlpLarge
		& \wlpLarge & \wlpLarge & \wlpLarge & \wlpLarge & \wlpLarge
		& \wlpLarge & \wlpLarge & \wlpLarge & \wlpLarge & \wlpLarge\\
		\hline
		13
		& \wlpNo & \wlpSixteen & \wlpYes & \wlpLarge & \wlpLarge
		& \wlpLarge & \wlpLarge & \wlpLarge & \wlpLarge & \wlpLarge
		& \wlpLarge & \wlpLarge & \wlpLarge & \wlpLarge & \wlpLarge\\
		\hline
		14
		& \wlpSixteen & \wlpYes & \wlpLarge & \wlpLarge & \wlpLarge
		& \wlpLarge & \wlpLarge & \wlpLarge & \wlpLarge & \wlpLarge
		& \wlpLarge & \wlpLarge & \wlpLarge & \wlpLarge & \wlpLarge\\
		\hline
		15
		& \wlpNo & \wlpLarge & \wlpLarge & \wlpLarge & \wlpLarge
		& \wlpLarge & \wlpLarge & \wlpLarge & \wlpLarge & \wlpLarge
		& \wlpLarge & \wlpLarge & \wlpLarge & \wlpLarge & \wlpLarge\\
		\hline
		16
		& \wlpLarge & \wlpLarge & \wlpLarge & \wlpLarge & \wlpLarge
		& \wlpLarge & \wlpLarge & \wlpLarge & \wlpLarge & \wlpLarge
		& \wlpLarge & \wlpLarge & \wlpLarge & \wlpLarge & \wlpLarge\\
		\hline
		17
		& \wlpLarge & \wlpLarge & \wlpLarge & \wlpLarge & \wlpLarge
		& \wlpLarge & \wlpLarge & \wlpLarge & \wlpLarge & \wlpLarge
		& \wlpLarge & \wlpLarge & \wlpLarge & \wlpLarge & \wlpLarge\\
		\hline
	\end{tabular}
	\endgroup
	
	\medskip
	\begin{minipage}{\textwidth}
		\small
		\raggedright
		
		\hspace*{1.5cm}\wlpYes:
		$A(T_{m,n})$ has the WLP by
		Proposition~\ref{pro:remaining-computational-cases}.
		\par\smallskip
		
		\hspace*{1.5cm}\wlpNo:
		$A(T_{m,n})$ fails the WLP by
		Proposition~\ref{pro:remaining-computational-cases}.
		\par\smallskip
		
		\hspace*{1.5cm}\wlpLarge:
		$A(T_{m,n})$ fails the WLP for $m+n\geq18$ by
		Proposition~\ref{pro:large-total-size}.
		\par\smallskip
		
		\hspace*{1.5cm}\wlpSixteen:
		$A(T_{m,n})$ fails the WLP for $m+n=16$ by
		Proposition~\ref{lem:special-sums}.
		
	\end{minipage}
\end{table}
The preceding results reduce the classification to the pairs in the
finite range $4\leq m\leq17$, $2\leq n\leq16$ satisfying
$m+n\leq15$ or $m+n=17$. We determine these finitely many cases through computations in
\texttt{Macaulay2}~\cite{GS}; the code is available
in~\cite{HPS26M2}.

\begin{proposition}\label{pro:remaining-computational-cases}
Assume that $4\leq m\leq17$, $2\leq n\leq16$, and either
$m+n\leq15$ or $m+n=17$. Then $A(T_{m,n})$ has the WLP if and only if
\begin{align*}
(m,n)\in{}&
\{4\}\times\{2,3,4,5,6,7,9,10,13\}\\
&\cup\{5,8\}\times\{2,3,5,6,9\}\\
&\cup\{6,10\}\times\{3,4,7\}\\
&\cup\{7\}\times\{2,3,4,7\}\\
&\cup\{(9,4),(11,2),(11,3),(11,6),(13,4),(14,3)\}.
\end{align*}
\end{proposition}

\begin{proof}
By Proposition~\ref{Proposition2.5}, it suffices to consider
multiplication by the sum $\ell$ of all the variables. For each pair
$(m,n)$ in the stated range, we compute the matrices of
\begin{equation*}
\times\ell:
[A(T_{m,n})]_j\longrightarrow[A(T_{m,n})]_{j+1}
\end{equation*}
in every degree $j$ and determine their ranks over $\mathbb{Q}$.
Computations in \texttt{Macaulay2}~\cite{GS}, using the code
in~\cite{HPS26M2}, show that these maps have maximal rank exactly for
the pairs listed in the statement and summarized in
Table~\ref{tab:Tadpole_WLP}.
The matrices have integer entries, so their ranks over an arbitrary
field of characteristic zero agree with their ranks over
$\mathbb{Q}$. Therefore, the same classification holds over every
field of characteristic zero.
\end{proof}
We are now ready to complete the proof of the main theorem.

\begin{proof}[Proof of Theorem~\ref{thm:mainTheorem_WLP}]
	Proposition~\ref{pro:WLP_small_m_and_n} gives the
	classification when $m=3$ or $n=1$. We may therefore
	assume that $m\geq4$ and $n\geq2$.
	
	If $m+n\geq18$, then
	Proposition~\ref{pro:large-total-size} shows that
	$A(T_{m,n})$ does not have the WLP.
	The case $m+n=16$ is excluded by
	Proposition~\ref{lem:special-sums}.
	
	It remains to consider $m+n\leq15$ or $m+n=17$.
	These pairs lie in the range of
	Proposition~\ref{pro:remaining-computational-cases},
	which gives their complete classification.
	
	Combining these cases with those in
	Proposition~\ref{pro:WLP_small_m_and_n} yields precisely
	the pairs listed in the statement.
\end{proof}

\section*{Acknowledgments}
Part of this work was done while the first author was visiting the Vietnam Institute for Advanced Study in Mathematics (VIASM), and he would like to thank the VIASM for hospitality and financial support.


\end{document}